\documentclass[reqno]{amsart}
\usepackage[scale=0.75, centering, headheight=14pt]{geometry}
\usepackage[latin1]{inputenc}
\usepackage[T1]{fontenc}
\usepackage{lmodern}
\usepackage[english]{babel}
\usepackage{microtype}
\usepackage{tikz}
\usepackage{amsmath,amssymb,amsfonts,amsthm}
\usepackage{mathtools,accents}
\usepackage{mathrsfs}
\usepackage{aliascnt}
\usepackage{braket}
\usepackage{bm}
\usepackage[initials]{amsrefs}
\allowdisplaybreaks
\usepackage[citecolor=blue,colorlinks]{hyperref}
\addto\extrasenglish{}

\usepackage{enumerate}
\usepackage{xcolor}

\usepackage{aliascnt}

\makeatletter
\def\newaliasedtheorem#1[#2]#3{
  \newaliascnt{#1@alt}{#2}
  \newtheorem{#1}[#1@alt]{#3}
  \expandafter\newcommand\csname #1@altname\endcsname{#3}
}
\makeatother

\theoremstyle{plain}
\newtheorem{theorem}{Theorem}[section]
\newaliasedtheorem{lemma}[theorem]{Lemma}
\newaliasedtheorem{prop}[theorem]{Proposition}
\newaliasedtheorem{claim}[theorem]{Claim}
\newaliasedtheorem{corollary}[theorem]{Corollary}

\theoremstyle{definition}
\newaliasedtheorem{definition}[theorem]{Definition}
\newaliasedtheorem{example}[theorem]{Example}

\theoremstyle{remark}
\newaliasedtheorem{remark}[theorem]{Remark}

\numberwithin{equation}{section}

\def\eps{\varepsilon}

\def\R{\mathbb R}
\def\Z{{\mathbb Z}}
\def\T{{\mathbb T}}

\DeclareMathOperator{\diver}{div}

\title[Regularity of Euler via interpolation theory]{Regularity results for rough solutions of the incompressible Euler equations via interpolation methods
}
\author[M. Colombo, L. De Rosa, L. Forcella]{Maria Colombo, Luigi De Rosa, \and Luigi Forcella}
\address{Maria Colombo  
\hfill\break  \'Ecole Polytechnique F\'ed\'erale de Lausanne, Institute of Mathematics, Station 8, CH-1015 Lausanne, Switzerland.}
\email{maria.colombo@epfl.ch}
\address{Luigi De Rosa 
\hfill\break  \'Ecole Polytechnique F\'ed\'erale de Lausanne, Institute of Mathematics, Station 8, CH-1015 Lausanne, Switzerland.}
\email{luigi.derosa@epfl.ch}
\address{Luigi Forcella 
\hfill\break  \'Ecole Polytechnique F\'ed\'erale de Lausanne, Institute of Mathematics, Station 8, CH-1015 Lausanne, Switzerland.}
\email{luigi.forcella@epfl.ch}



\begin{document}

\maketitle

\begin{abstract}
Given any solution $u$ of the Euler equations which is assumed to have some regularity in space -- in terms of Besov norms, natural in this context -- we show by interpolation methods that it enjoys a corresponding regularity in time and that the associated pressure $p$  is twice as regular as $u$. This generalizes a recent result by Isett \cite{Is2013} (see also Colombo and De Rosa \cite{CD18}), which covers the case of H\"older spaces.
\end{abstract}

\section{Introduction}\label{sec:intro}

In the spatial periodic setting $\T^3=\R^3 / \Z^3$, we consider the incompressible Euler equations 
\begin{equation}\label{E}
\left\{\begin{array}{l}
\partial_t u+  \diver(u \otimes u)+\nabla p =0\\ 
 \diver u = 0\,
\end{array}\right.\qquad \mbox{in } (0,T) \times \T^3
\end{equation}
where $u: (0,T)\times \T^3  \rightarrow \R^3$ represents the velocity of an incompressible fluid, $p:(0,T)\times \T^3 \rightarrow \R$ is the hydrodynamic pressure, with the constraint $\int_{\T^3}p\, dx =0$, which guaranties its uniqueness.

The interest for low-regularity solutions to the Euler equations is strongly related to Kolmogorov's 1941 theory of turbulence \cite{K41} and to the Onsager's conjecture \cite{Ons49}. In recent years, distributional solutions belonging to H\"older spaces were built with convex integration techniques, starting from the works of De Lellis and Sz\'{e}kelyhidi  \cite{DS2013,DS2014} and leading to the proof of the Onsager's conjecture by Isett, see \cite{Is2018}.

Such techniques were recently adapted by Buckmaster and Vicol to the Navier-Stokes equations by developing a Sobolev (rather than H\"older) based method \cite{BV2017}, which in turn appears also a recent work by Modena and Sz\'{e}kelyhidi in \cite{MoSz}.
For instance, in the context of the physical theory of intermittency it is currently an open problem (see  \cite[Open Problem 5]{BV2019}) to determine the best exponent $\theta$ such that $L^\infty((0,T); H^{\theta}(\T^3))$ solutions conserve the energy (it is known that for $\theta= 5/6$ conservation holds).


The following theorem provides a regularization property of the Euler equations, for solutions which enjoy some a priori Sobolev or Besov regularity in space. Roughly speaking, we prove that the pressure associated to any such solution enjoys double regularity in space with respect to $u$, and that both $u$ and $p$ enjoy a corresponding time regularity. In the main theorem  below, by $B^\theta_{s,\infty}$ we denote a Besov space, rigorously defined in \autoref{sec:tool}. The choice to work in these spaces is motivated to avoid an $\eps$-loss of regularity in time. 


\begin{theorem}\label{t:main}
Let $(u,p)$ be a distributional solution to \eqref{E} in $(0,T)\times \T^3$, for some $T<\infty$. For any $\theta\in (0,1)$, $s\in [1,\infty]$, $r\in (1,\infty)$, the following implications are true:
\begin{itemize}
\item[$(i)$] if $u\in L^{2s}((0,T);B^\theta_{2r,\infty}(\T^3))$, then $u\in B^\theta_{s,\infty}((0,T);L^r(\T^3))$ and  $p\in L^{s}((0,T);B^{2\theta}_{r,\infty}(\T^3))$;
\item[$(ii)$] if $u\in L^{3s}((0,T);B^\theta_{4r,\infty}(\T^3))$ and $\theta >1/2$, then $p\in B^{2\theta-1-\beta}_{s,\infty}((0,T);B^{1+\beta}_{r,\infty}(\T^3))$ for any $\beta \in [0,2\theta-1)$;
\item[$(iii)$] if $u\in  L^{3s}((0,T);B^\theta_{3r,\infty}(\T^3))$ and if  $\theta \leq 1/2$, then   $p\in B^{2\theta-\varepsilon}_{s,\infty}((0,T);L^r(\T^3)),$ for any $\varepsilon>0$. Moreover in the case $\theta >1/2$ we have  $ p\in W^{1,s}((0,T);B^{2\theta-1}_{r,\infty}(\T^3))$;
\item[$(iv)$]if $u\in  L^{6s}((0,T);B^\theta_{6r,\infty}(\T^3))$ and  $\theta>1/2$, then $ \partial_t p\in B^{2\theta-1-\varepsilon}_{s,\infty}((0,T);L^r(\T^3)),$ for any $\varepsilon>0$.
\end{itemize}
\end{theorem}

Then we obtain the following corollary on the Sobolev solutions by considering suitable embeddings between Sobolev and Besov spaces.
\begin{corollary}\label{t:main-easy}
Let $(u,p)$ be a distributional solution to \eqref{E} in $(0,T)\times \T^3$, for some $T<\infty$. For any $\theta\in (0,1)$, $s\in [1,\infty]$, $r\in (1,\infty)$, the following implications hold true:
\begin{itemize}
\item[$(i)$] if $u\in L^{2s}((0,T);W^{\theta,2r}(\T^3))$, then 
$u\in W^{\theta-\eps,s}((0,T);L^r(\T^3))$ and $ p\in L^{s}((0,T);W^{2\theta-\eps,r}(\T^3));$
\item[$(ii)$] if $\theta \leq 1/2$ and $u\in  L^{3s}((0,T); W^{\theta,3r}(\T^3))$, or if $\theta>1/2$ and $u\in  L^{6s}((0,T);W^{\theta,6r}(\T^3))$,
then  $p\in W^{2\theta-\varepsilon,s}((0,T);L^r(\T^3)).$
\end{itemize}
\end{corollary}

When $s=r=\infty$, identifying $W^{\theta,\infty}$ with the corresponding H\"older space, the previous theorem corresponds formally to \cite[Theorem 1.1]{Is2013} and \cite[Theorem 1.1]{CD18}: roughly speaking, it says that if $(u,p) $ is a distributional solution to \eqref{E}, $\theta\in (0,1)$ and $u\in  L^{\infty}((0,T); C^{\theta}(\T^3))$, then $u\in C^{\theta-\eps}((0,T);L^\infty(\T^3)),$ namely $u \in C^{\theta-\varepsilon}((0,T)\times \T^3)$ and $p\in C^{2\theta-\varepsilon}((0,T)\times \T^3).$\\

\autoref{t:main} follows from two main ingredients: on one side, we obtain the time regularity by estimating, for any time increment $h$, some norm $||u(t+h)-u(t)||$ by comparison between $u$ and the convolution of $u$ with a mollification kernel at some scale $\delta$, which is then linked to $h$. 
On the other side, to obtain the double regularity of the pressure we look at 
\begin{equation}
\label{eqn:p}
-\Delta p =\diver \diver (u \otimes u),
\end{equation}
which is the formal equation solved by $p.$ We consider a bilinear operator which associates to two divergence-free vector fields $(u,v)$ the solution to $-\Delta p= \diver \diver (u \otimes v)$ and we apply an abstract interpolation result for bilinear operators (see \autoref{thm:tri} below).
Previous results on the regularity of the pressure in H\"older spaces (see \cite{C2014} and \cite{CD18}) were instead based on suitable representation formulas for the pressure by means of the Green kernel of the Laplacian, while this strategy using real interpolation methods seems to be new in this context.

\section{Preliminary tools and notations}\label{sec:tool}
Along the paper, we will consider $\T^3$ as spatial domain, identifying it with the 3-dimensional cube $[0,1]^3 \subset \R^3 $. Thus for any $f:\T^3\to\R^3$ we will  always work with its periodic extension to the whole space.\\ 

We will define the norms for a domain $\Omega\subseteq \R^d$, for a general dimension $d\geq 1$, since in this way we can handle both the space and the rime regularities. Let $\Omega\subseteq \R^d$ be an open and Lipschitz domain. For $\theta \in (0,\infty)$, $r,s\in[1,\infty]$, the $L^r(\Omega)$ and $W^{\theta,r}(\Omega)$ spaces are the classical Lebesgue and Sobolev-Slobodeckij spaces, with the usual identifications $W^{0,r}(\Omega)=L^r(\Omega)$ and $W^{\theta,\infty}(\Omega)=C^\theta(\Omega)$. We first define the Besov spaces on the whole $\R^d$, then their version on general open sets $\Omega$ will be defined by extension. For any $\theta\in(0,\infty)$, let $\theta^-$ to be the biggest integer which is strictly less than $\theta$. For any non integer $\theta \in(0,\infty)$, the  Besov space $B^\theta_{r,s}(\R^d)$ is the space of functions $f\in W^{\theta^-, r} (\R^d) $ such that
\begin{align*}
[f]_{B^\theta_{r,s}(\R^d)}=\sum_{|\alpha|=  \theta^-}\left(\int_{\R^d}\frac{1}{|h|^{d+(\theta-\theta^-)s}}\left( \int_{\R^d} |D^\alpha f(x+h)-D^\alpha f(x)|^r\,dx\right)^{\frac{s}{r}}\,dh\right)^{\frac{1}{s}}<\infty,
\end{align*}
with the usual generalization when $r,s=\infty$. The full Besov norm will be then given by 
\[
\|f\|_{B^\theta_{r,s}(\R^d)}=\|f\|_{W^{\theta^-, r} (\R^d) }+[f]_{B^\theta_{r,s}(\R^d)}.
\]
If instead $\theta>0$ is an integer, the Besov space $B^\theta_{r,s}(\R^d)$ consists of all the functions $f\in W^{\theta, r}(\R^d)$, such that
\[
[f]_{B^\theta_{r,s}(\R^d)}=\sum_{|\alpha|=\theta}\left(\int_{\R^d}\frac{1}{|h|^{d+s}}\left( \int_{\R^d} |D^\alpha f(x+2h)-2D^\alpha f(x+h)+D^\alpha f(x) |^r\,dx\right)^{\frac{s}{r}}\,dh\right)^{\frac{1}{s}}<\infty,
\]
again with the usual generalization when $r,s=\infty$. Thus the full norm will be given by
\[
\|f\|_{B^\theta_{r,s}(\R^d)}=\|f\|_{W^{ \theta, r} (\R^d) }+[f]_{B^\theta_{r,s}(\R^d)}.
\]
For any open and Lipschitz set $\Omega$ we then define
\[
B^\theta_{r,s}(\Omega)=\left\{f:\Omega\rightarrow \R^d\, \, \hbox{s.t.}\, \, \exists\,\, \tilde f \in B^\theta_{r,s}(\R^d),\, \, \tilde f|_{\Omega}=f \right\},
\]
where the semi-norm is given by 
\[
[f]_{B^\theta_{r,s}(\Omega)}=\inf \left\{ [\tilde f]_{B^\theta_{r,s}(\R^d)},\, \, \tilde f|_{\Omega}=f   \right\}.
\]
By the definitions above we have that for any non integer $\theta\in (0,\infty)$, $B^\theta_{r,r}(\Omega)=W^{\theta,r}(\Omega)$ for any $r\in [1,\infty]$, which in the case $r=\infty$ gives $B^\theta_{\infty,\infty}(\Omega)=C^\theta(\Omega)$.
Moreover, since the domain $\Omega$ is Lipschitz, we always have the existence of a linear extension operator to the whole space. It is well know that this operator turns out to be also continuous between every Sobolev or Besov spaces. \\

Considering the flat $d$-dimensional torus $\T^d$, we define the Besov norm as above with $\Omega=[0,4]^d$ that is, we compute the norm in $4$ copies of $\T^d$. 
\\

Dealing with time dependent vector fields $u=u(t,x)$, we will use the notations $[u(t)]$ and $\|u(t)\|$ when the spatial semi-norm or norm, respectively, are computed at the fixed time $t$.\\

\noindent We give the following interpolation result in Besov spaces.
\begin{prop}
Let $\Omega\subset \R^d$ be an open  and Lipschitz set. For any $r\in [1,\infty]$, $\theta, \gamma \in (0,1)$ with $\theta\geq \gamma$, there exists a constant $C>0$ such that
\begin{align}\label{interp_besov}
[f]_{B^\gamma_{r,\infty}(\Omega)}&\leq C \|f\|_{L^r(\Omega)}^{1-\frac{\gamma}{\theta}}\|f\|_{B^\theta_{r,\infty}(\Omega)}^{\frac{\gamma}{\theta}},\\\label{interp_besov_1}
[f]_{B^\theta_{r,\infty}(\Omega)}&\leq C \|f\|_{B^\gamma_{r,\infty}(\Omega)}^{\frac{1-\theta}{1-\gamma}}\| f\|_{W^{1,r}(\Omega)}^{\frac{\theta-\gamma}{1-\gamma}}.
\end{align}
\end{prop} 
\noindent Note that the same inequalities hold if one replaces all the semi-norms with the full norms.
\begin{proof}
We start by proving \eqref{interp_besov} and \eqref{interp_besov_1} in the whole space $\R^d$. Note that for every $f \in B^\theta_{r,\infty}(\R^d)$ and $\theta\geq \gamma$, we have 
\begin{equation}\label{besov_1}
[f]_{B^\gamma_{r,\infty}(\R^d)}\leq 2\left( \|f\|_{L^r(\R^d)}+[f]_{B^\theta_{r,\infty}(\R^d)}\right).
\end{equation}
By plugging in \eqref{besov_1} the rescaled function $f(\varepsilon x)$, we also get 
\[
\varepsilon^\gamma [f]_{B^\gamma_{r,\infty}(\R^d)}\leq 2\left( \|f\|_{L^r(\R^d)}+\varepsilon^\theta [f]_{B^\theta_{r,\infty}(\R^d)}\right),
\]
for every $\varepsilon>0$. Thus by choosing $\varepsilon=\|f\|_{L^r(\R^d)}^{\frac{1}{\theta}}[f]^{-\frac{1}{\theta}}_{B^\theta_{r,\infty}(\R^d)}$, we get
\eqref{interp_besov} for $\Omega = \R^d$. 
Take now $\lambda\in [0,1)$ such that $(1-\lambda)\gamma+\lambda=\theta$. We estimate
\begin{align*}
\frac{\|f(\cdot+y) -f(\cdot)\|_{L^r(\R^d)}}{|y|^\theta}&=\left(\frac{\|f(\cdot+y) -f(\cdot)\|_{L^r(\R^d)}}{|y|^\gamma}\right)^{1-\lambda}\left(\frac{\|f(\cdot+y) -f(\cdot)\|_{L^r(\R^d)}}{|y|}\right)^{\lambda}\\
&\leq [f]_{B^\gamma_{r,\infty}(\R^d)}^{1-\lambda}\| \nabla f \|_{L^r(\R^d)}^\lambda,
\end{align*}
from which, since $\lambda=\frac{\theta-\gamma}{1-\gamma}$,   we conclude
\eqref{interp_besov_1} for $\Omega = \R^d$. 
If $f\in B^\theta_{r,\infty}(\Omega)$ for $\Omega$ as in the statement, \eqref{interp_besov} and \eqref{interp_besov_1} easily follow from their versions in $\R^d$ and the existence of a (continuous) extension operator.
\end{proof}

Let $\varphi\in\mathcal C^{\infty}_c(\R^d)$ a smooth, nonnegative and compactly supported function with $\|\varphi\|_{L^1}=1.$  For any $\delta>0$ we define $\varphi_\delta(x)=\delta^{-d}\varphi(x/\delta)$ and we consider, for any vector field $f:\T^d\to\R^d$ its regularization 
$
f_\delta(x)=(f\ast\varphi_\delta)(x)=\int_{\R^d}f(x-y)\varphi_\delta(y)\,dy.
$ We conclude this section by recalling some classical estimates. The third one is for instance the one used in \cite{CET94} to prove the positive statement of the Onsager's conjecture.
\begin{prop} For any $f: \T^d\to \R^d,$ $\theta\in(0,1),$ $r\in[1,\infty]$ and any integer $n\geq 0$, we have the following
\begin{align}\label{molli:1}
\|f-f_\delta\|_{L^r(\T^d)}&\leq C\delta^\theta\|f\|_{B^\theta_{r,\infty}(\T^d)},\\\label{molli:2}
\|f_\delta\|_{W^{n+1,r}(\T^d)}&\leq C\delta^{\theta-n-1}\|f\|_{B^\theta_{r,\infty}(\T^d)},\\\label{molli:3}
\|f_\delta\otimes f_\delta-(f\otimes f)_\delta\|_{W^{n,r}(\T^d)}&\leq C\delta^{2\theta-n}\|f\|^2_{B^{\theta}_{2r,\infty}(\T^d)},
\end{align}
for some constant $C>0$ depending on $\theta,r, n$ but otherwise independent of $\delta$.
\end{prop}

\section{Abstract multilinear interpolation}\label{sec:multi}
In this section we provide some estimates for multilinear operators, by means of abstract real interpolation methods. They are the core of the paper, and the proof of \autoref{t:main} relies on them. We start by recalling some definitions and basic facts about interpolation spaces and we refer the reader to the classical monographs \cite{BL,Lun,Tri} for further details.
\\

Let $(X,\|\cdot\|_X)$ and $(Y,\|\cdot\|_Y)$ be two real Banach spaces. The couple $(X,Y)$ is said to be an interpolation  couple if both $X$ and $Y$ are continuously embedded in a topological Hausdorff vector space. 
For any interval $I\subseteq(0,\infty)$ we denote by $L^r_*(I)$ the Lebesgue space of $r$-summable functions with respect to the measure $dt/t.$ Let use notice that in particular $L^\infty(I)=L^\infty_*(I).$ Moreover, we recall the definition of the $K$-function, by introducing the following notation. Given $x\in X+Y$ we denote $\Omega(x)=\{  (a,b)\in X\times Y : a+b=x\}\subset X\times Y.$
\begin{definition}
For every $x\in X+Y$  and $t>0,$ the $K$-function is defined by  
\begin{equation}\label{K:fun}
K(t,x,X,Y)=\inf_{\Omega(x)}\{\|a\|_X+t\|b\|_Y\}.
\end{equation}
If no confusion can occur, we simply write $K(t,x)$ instead of $K(t,x,X,Y).$
\end{definition}
\begin{definition}
Let $\theta\in(0,1)$ and $r\in[1,\infty].$ We set
\[
(X,Y)_{\theta,r}=\left\{x\in X+Y \hbox{ s.t. } t\mapsto t^{-\theta}K(t,x)
\in L^r_*(0,\infty)\right\}
\]
endowed with the norm 
\begin{equation*}
\|x\|_{(X,Y)_{\theta,r}}=\|t^{-\theta}K(\cdot,x)\|_{L^r_*}.
\end{equation*}
\end{definition}
\noindent For these spaces we have the following inclusions 
\begin{equation}\label{interp_inclus}
X\cap Y \hookrightarrow (X,Y)_{\theta,r}\hookrightarrow (X,Y)_{\theta,s}\hookrightarrow X+Y,
\end{equation}
for every $\theta \in (0,1)$ and $r,s\in [1,\infty]$ with $r\leq s$. Moreover if $\gamma\geq \theta$ we also have $(X,Y)_{\gamma,r}\hookrightarrow (X,Y)_{\theta,s}$, for every $r,s\in [1,\infty]$.
\noindent The following two remarks will be useful in the proof of  \autoref{thm:tri}.
\begin{remark}\label{K_fun_new}
When $Y \hookrightarrow X$, the definition of $K$ in \eqref{K:fun} does not change if instead of $\Omega(x)$ we  consider the set $\tilde\Omega(x)=\{(a,b)\in \Omega(x) \hbox { s.t. } \|a\|_{X}\leq \|x\|_{X} \} $; in other words,
\[
K(t,x,X,Y)=\inf_{\Omega(x)}\{\|a\|_X+t\|b\|_Y\}=\inf_{\tilde \Omega(x)}\{\|a\|_X+t\|b\|_Y\}.
\]
Indeed, since $Y\hookrightarrow X$, one can choose $a=x$ and $b=0$ in \eqref{K:fun}, obtaining $K(t,x)\leq\|x\|_{X}  $. On the other hand, we have that 
$\|a\|_X+t\|b\|_Y>\|x\|_{X}$ for all $(a,b)\in \tilde\Omega(x)^c$.
\end{remark}
\begin{remark}\label{restric}
Consider again the case $ Y\hookrightarrow X$. Since $a+b=x$, we have
\[
\| a\|_X+\|b\|_X\leq 2 \|a\|_X +\|x\|_X\leq 3\|x\|_X, \qquad \forall\, {(a,b) \in \tilde \Omega(x)}.
\]
\end{remark}
\noindent It is well known that $\left((X,Y)_{\theta,r},\|\cdot\|_{(X,Y)_{\theta,r}}\right)$ is a Banach space. Furthermore, we recall that a linear operator $T$ behaves nicely with respect to interpolation, i.e. if $T\in \mathcal L(X_1,Y_1)\cap \mathcal L(X_2,Y_2),$  then  $T\in \mathcal L((X_1,X_2)_{\theta,r},(Y_1,Y_2)_{\theta,r})$ for any $\theta\in(0,1)$ and $r\in[1,\infty].$ \\

Instead of linear operators, our aim is to treat the case of multilinear operators, in particular bilinear and trilinear ones.  It is worth mentioning that there exists a wide literature on Interpolation Theory for multilinear operators, see for example the works \cite{BL}, \cite{GM}, \cite{LP} and  \cite{Ma}, but at the best of our knowledge the following results are new. We also emphasise that they are precisely designed for the applications to incompressible fluid models of the next section.  In what follows, a conjugate pair $(s,s^\prime)$ is a couple of reals satisfying $s^\prime=\frac{s}{s-1}.$

\begin{theorem}\label{bil:thm}
Let $(X_1,X_2)$ and $(Y_1,Y_2)$ be two interpolation couples. Let $T$ be a bilinear operator satisfying 
\begin{align}\label{bil:est:1}
\|T(a_1,a_2)\|_{Y_1}&\leq C_0\|a_1\|_{X_1}\|a_2\|_{X_1},\\\label{bil:est:2}
\|T(b_1,b_2)\|_{Y_2}&\leq C_0\|b_1\|_{X_2}\|b_2\|_{X_2},
\end{align}
and
\begin{equation}\label{bil:est:3}
\|T(a,b)\|_{(Y_1,Y_2)_{\frac12,\infty}}+\|T(b,a)\|_{(Y_1,Y_2)_{\frac12,\infty}}\leq C_0\|a\|_{X_1}\|b\|_{X_2},
\end{equation}
for some constant $C_0>0$ independent on $a,a_1,a_2\in X_1$ and $b,b_1,b_2\in X_2,$ where we implicitly assume that $T$ is well defined between the spaces involved in the previous estimates. Then, for any $\theta,\gamma\in(0,1),$  $r,s,s^\prime \in[1,\infty]$ with $s,s^\prime$ being a conjugate pair,
\begin{equation*}
\|T(x_1,x_2)\|_{(Y_1,Y_2)_{\frac{\theta+\gamma}{2},r}}\leq  C_0\|x_1\|_{(X_1,X_2)_{\gamma,rs}} \|x_2\|_{(X_1,X_2)_{\theta,rs^\prime}} \qquad \forall\, x_1 \in (X_1,X_2)_{\gamma,rs}, \,\forall\, x_2 \in (X_1,X_2)_{\theta,rs^\prime}.
\end{equation*}
In particular, for $\gamma=\theta$ and $s=s^\prime=2,$ we get 
\begin{equation*}
\|T(x,x)\|_{(Y_1,Y_2)_{\theta,r}}\leq C_0\|x\|^2_{(X_1,X_2)_{\theta,2r}}, \quad \forall\, x\in(X_1,X_2)_{\theta,2r}.
\end{equation*}
\end{theorem}
\begin{proof}
Let $x_1\in (X_1,X_2)_{\gamma,sr}$ and $x_2\in (X_1,X_2)_{\theta,rs^\prime}.$ Then we can write $x_1=a_1+b_1$ and $x_2=a_2+b_2$ for some $a_1,a_2\in X_1$ and $b_1,b_2\in X_2,$ by definition. Since $T$ is bilinear we have 
\[
T(x_1,x_2)=T(a_1,a_2)+T(a_1,b_2)+T(b_1,a_2)+T(b_1,b_2).
\]
From \eqref{bil:est:3} we know that $T(a_1,b_2)\in(Y_1,Y_2)_{\frac12,\infty},$ hence for any $t, \eps>0$ there exist $T_1\in Y_1$ and  $T_2\in Y_2$ 
such that $T(a_1,b_2)=T_1+T_2$ and
\begin{equation}\label{bil:est:4}
\begin{aligned}
\|T_1\|_{Y_1}+t\|T_2\|_{Y_2}&\leq (1+\eps)K(t,T(a_1,b_2),Y_1,Y_2)\\
&\leq (1+\varepsilon)\sqrt t\|T(a_1,b_2)\|_{(Y_1,Y_2)_{\frac12,\infty}}\leq (1+\eps)C_0\sqrt t\|a_1\|_{X_1}\|b_2\|_{X_2}.
\end{aligned}
\end{equation}
Similarly, we can decompose $T(b_1,a_2)=U_1+U_2$ with $U_1\in Y_1$ and  $U_2\in Y_2$ with estimate 
\begin{equation}\label{bil:est:5}
\|U_1\|_{Y_1}+t\|U_2\|_{Y_2}\leq (1+\varepsilon) C_0\sqrt t\|a_1\|_{X_1}\|b_2\|_{X_2}.
\end{equation}
Therefore we can write $T(x_1,x_2)=V+W$, where 
\[
\begin{aligned}
V&=T(a_1,a_2)+T_1+U_1\in Y_1,\\
W&=T(b_1,b_2)+T_2+U_2\in Y_2.
\end{aligned}
\]
Summing up \eqref{bil:est:1}
--\eqref{bil:est:5} yields to 
\[
\begin{aligned}
\|V\|_{Y_1}+t\|W\|_{Y_2}&\leq (1+\eps)C_0\left(\|a_1\|_{X_1}\|a_2\|_{X_1}+\sqrt t\left(\|a_1\|_{X_1}\|b_2\|_{X_2}+\|a_2\|_{X_1}\|b_1\|_{X_2}\right)+t\|b_1\|_{X_2}\|b_2\|_{X_2}\right)\\
&=(1+\eps)C_0\left(\|a_1\|_{X_1}+\sqrt t\|b_1\|_{X_2}\right)\left(\|a_2\|_{X_1}+\sqrt t\|b_2\|_{X_2}\right),
\end{aligned}
\]
which in turn implies 
\begin{equation}\label{bil:est:6}
K(t,T(x_1,x_2),Y_1,Y_2)\leq (1+\eps)C_0K(\sqrt t,x_1,X_1,X_2)K(\sqrt t,x_2,X_1,X_2).
\end{equation}
Multiplying \eqref{bil:est:6} by $t^{-(\gamma+\theta)/2}$ and by taking the $L^r_*(0,\infty)$-norm we get, by means of the H\"older inequality with conjugate exponents $s$ and $s^\prime,$ 
\[
\begin{aligned}
\|T(x_1,x_2)\|_{(Y_1,Y_2)_{\frac{\theta+\gamma}{2},r}} &=
\|(\cdot)^{-(\theta+\gamma)/2}K(\cdot,T(x_1,x_2))\|_{L^r_*}
\\& \leq (1+\eps)C_0\left(\|(\cdot)^{-s\theta/2}K^s(\sqrt\cdot,x_1)\|^{1/s}_{L^r_*}\|(\cdot)^{-s^\prime\theta/2}K^{s^\prime}(\sqrt\cdot,x_2)\|^{1/{s^\prime}}_{L^r_*}\right)\\
&=(1+\eps)C_0\|x_1\|_{(X_1,X_2)_{\gamma,rs}}\|x_2\|_{(X_1,X_2)_{\theta,rs^\prime}},
\end{aligned}
\]
and since the last inequality holds true for any $\varepsilon>0$, we are done.
\end{proof}
Let us now focus on trilinear operators, for which a similar result as in \autoref{bil:thm} can be proved. In what follows, it will be useful to consider interpolation couples $(X_1, X_2)$ such that $X_2\hookrightarrow X_1.$ For  sake of  clarity, we require that the trilinear operator in the statement is symmetric in each variable, even though a suitable adaptation would work without this requirement. 
\begin{theorem}\label{thm:tri}
Let $C_0>0$, $(X_1,X_2)$ and $(Y_1,Y_2)$ be two interpolation couples with $X_2\hookrightarrow X_1.$ Let $T$ be a trilinear and symmetric operator satisfying the following conditions
\begin{equation}\label{tri:est:1}
\|T(a_1,a_2,a_3)\|_{Y_1}\leq C_0\|a_1\|_{X_1}\|a_2\|_{X_1}\|a_3\|_{X_1},
\end{equation}
\begin{equation}\label{tri:est:2}
\|T(b_1,b_2,b_3)\|_{Y_2}\leq C_0\Big(\|b_1\|_{X_1}\|b_2\|_{X_2}\|b_3\|_{X_2}+\|b_1\|_{X_2}\|b_2\|_{X_1}\|b_3\|_{X_2}+\|b_1\|_{X_2}\|b_2\|_{X_2}\|b_3\|_{X_1}\Big),
\end{equation}
and
\begin{equation}\label{tri:est:3}
\|T(a_1,b_2,b_3)\|_{(Y_1,Y_2)_{\frac12,\infty}}\leq C_0\|a_1\|_{X_1}\Big(\|b_2\|_{X_2}\|b_3\|_{X_1}+\|b_2\|_{X_1}\|b_3\|_{X_2}\Big),
\end{equation}
 where we implicitly assume that $T$ is well defined between the spaces involved in the previous estimates. Then for any $\gamma,\theta\in(0,1)$ and $r,s\in[1,\infty],$ 
for every $x_1,x_2,x_3 $ we have
\begin{equation}\label{eqn:ts-abstr-interp}
\begin{aligned}
\|T(x_1,x_2,x_3)\|_{(Y_1,Y_2)_{\frac{\theta+\gamma}{2},r}}&\leq 3 C_0\bigg(\|x_1\|_{X_1}\|x_2\|_{(X_1,X_2)_{\gamma,rs}}\|x_3\|_{(X_1,X_2)_{\theta,rs^\prime}}\\
&\quad+\|x_1\|_{(X_1,X_2)_{\gamma,rs}} \left(\|x_2\|_{X_1}\|x_3\|_{(X_1,X_2)_{\theta,rs^\prime}}+\|x_2\|_{(X_1,X_2)_{\theta,rs^\prime}}\|x_3\|_{X_1} \right) \bigg).
\end{aligned}
\end{equation}
In particular, for $\gamma=\theta$ and $s=s^\prime=2,$ we get 
\begin{equation*}
\|T(x,x,x)\|_{(Y_1,Y_2)_{\theta,r}}\leq 3 C_0\|x\|_{X_1}\|x\|^2_{(X_1,X_2)_{\theta,2r}}, \quad \forall x\,\in(X_1,X_2)_{\theta,2r}.
\end{equation*}

\end{theorem}
\begin{proof}
We assume without loss of generality that $\theta\geq\gamma.$ Consider $x_1\in(X_1,X_2)_{\gamma,rs}$ and $x_2,x_3\in(X_1,X_2)_{\theta,rs^\prime}.$ For $k=1,2,3$ we write $x_k=a_k+b_k$ with $a_k\in X_1$ and $b_k\in X_2;$ therefore we expand
\[
T(x_1,x_2,x_3)=U+V+W
\]  
where 
\begin{align*}
U&=T(a_1,a_2,a_3)+T(b_1,a_2,a_3)+T(a_1,b_2,a_3)+T(a_1,a_2,b_3),\\
V&=T(b_1,b_2,a_3)+T(b_1,a_2,b_3)+T(a_1,b_2,b_3),\\
W&=T(b_1,b_2,b_3).
\end{align*}
Since $X_2\hookrightarrow X_1$ we have that $b_k\in X_1$ for any $k=1,2,3,$ then by \eqref{tri:est:1} we can control $U$ as 
\begin{equation}\label{tri:est:4}
\begin{aligned}
\|U\|_{Y_1}& \leq C_0\Big(\|a_1\|_{X_1}\|a_2\|_{X_1}\|a_3\|_{X_1}+\|b_1\|_{X_1}\|a_2\|_{X_1}\|a_3\|_{X_1}\\
&\quad+\|a_1\|_{X_1}\|b_2\|_{X_1}\|a_3\|_{X_1}+\|a_1\|_{X_1}\|a_2\|_{X_1}\|b_3\|_{X_1} \Big).
\end{aligned}
\end{equation}
The symmetry of the operator $T$ and \eqref{tri:est:3} imply that every term defining  $V$ belongs to $(Y_1,Y_2)_{\frac12,\infty}.$ Let us consider without loss of generality the term $T(b_1,b_2,a_3);$ as already done in \autoref{bil:thm}, for any $t,\varepsilon >0$ there exist $T_1\in Y_1$ and $T_2\in Y_2$ such that $T(b_1,b_2,a_3)=T_1+T_2$ and 
\[
\begin{aligned}
\|T_1\|_{Y_1}+t\|T_2\|_{Y_2}&\leq (1+\varepsilon)K(t,T(b_1,b_2,a_3),Y_1,Y_2)\leq (1+\varepsilon)\sqrt t\|T(b_1,b_2,a_3)\|_{(Y_1,Y_2)_{\frac12,\infty}}\\
&\leq (1+\varepsilon)C_0\sqrt t\|a_3\|_{X_1}\left(\|b_1\|_{X_1}\|b_2\|_{X_2}+\|b_1\|_{X_2}\|b_2\|_{X_1}\right).
\end{aligned}
\]
We point out that the elements $T_1$ and $T_2$ actually depend on $a_3,b_1,b_2, \varepsilon$ and $t$ as well.
The same consideration for the other two terms defining $V$ yields, for any $t,\varepsilon>0,$ to the existence of $V_1\in Y_1$ and $V_2\in Y_2$ such that $V=V_1+V_2$ and 
\begin{equation}\label{tri:est:5}
\begin{aligned}
\|V_1\|_{Y_1}&+t\|V_2\|_{Y_2}\leq (1+\varepsilon)C_0\sqrt t\Big(\|a_1\|_{X_1}\big(\|b_2\|_{X_1}\|b_3\|_{X_2}+\|b_2\|_{X_2}\|b_3\|_{X_1}\big)\\
&\quad+\|a_2\|_{X_1}\big(\|b_1\|_{X_1}\|b_3\|_{X_2}+\|b_1\|_{X_2}\|b_3\|_{X_1}\big)+\|a_3\|_{X_1}\big(\|b_1\|_{X_1}\|b_2\|_{X_2}+\|b_1\|_{X_2}\|b_2\|_{X_1}\big)\Big).
\end{aligned}
\end{equation}

\noindent By using \eqref{tri:est:2} we also get 

\begin{equation}\label{tri:est:6}
\begin{aligned}
\|T(b_1,b_2,b_3)\|_{Y_2}\leq C_0\Big(\|b_1\|_{X_1}\|b_2\|_{X_2}\|b_3\|_{X_2}+\|b_1\|_{X_2}\|b_2\|_{X_1}\|b_3\|_{X_2}+\|b_1\|_{X_2}\|b_2\|_{X_2}\|b_3\|_{X_1}\Big).
\end{aligned}
\end{equation}
By combining \eqref{tri:est:4}, \eqref{tri:est:5} and \eqref{tri:est:6} we obtain, for any $t,\varepsilon>0,$ a decomposition of $T(x_1,x_2,x_3)=(U+V_1)+(V_2+W),$ with $U+V_1\in Y_1$ and $V_2+W\in Y_2$ such that
\begin{equation*}
\begin{aligned}
\|U+V_1&\|_{Y_1}+t\|V_2+W\|_{Y_2}\leq (1+\varepsilon)C_0\bigg( \|a_1\|_{X_1}\|a_2\|_{X_1}\|a_3\|_{X_1}+\|b_1\|_{X_1}\|a_2\|_{X_1}\|a_3\|_{X_1}\\
&\quad+\|a_1\|_{X_1}\|b_2\|_{X_1}\|a_3\|_{X_1}+\|a_1\|_{X_1}\|a_2\|_{X_1}\|b_3\|_{X_1}+\sqrt t\Big(\|a_1\|_{X_1}\big(\|b_2\|_{X_1}\|b_3\|_{X_2}+\|b_2\|_{X_2}\|b_3\|_{X_1}\big)\\
&\quad+\|a_2\|_{X_1}\big(\|b_1\|_{X_1}\|b_3\|_{X_2}+\|b_1\|_{X_2}\|b_3\|_{X_1}\big)+\|a_3\|_{X_1}\big(\|b_1\|_{X_1}\|b_2\|_{X_2}+\|b_1\|_{X_2}\|b_2\|_{X_1}\big)\Big)\\
&\quad+t\Big(\|b_1\|_{X_1}\|b_2\|_{X_2}\|b_3\|_{X_2}+\|b_1\|_{X_2}\|b_2\|_{X_1}\|b_3\|_{X_2}+\|b_1\|_{X_2}\|b_2\|_{X_2}\|b_3\|_{X_1}\Big) \bigg)\\
&\leq (1+\varepsilon)C_0\Big( \big(\|a_1\|_{X_1}+\|b_1\|_{X_1}\big)\big(\|a_2\|_{X_1}+\sqrt t \|b_2\|_{X_2}\big)\big(\|a_3\|_{X_1}+\sqrt t \|b_3\|_{X_2}\big)\\
&\quad+\big(\|a_1\|_{X_1}+\sqrt t\|b_1\|_{X_2}\big)\big(\|a_2\|_{X_1}+ \|b_2\|_{X_1}\big)\big(\|a_3\|_{X_1}+\sqrt t \|b_3\|_{X_2}\big)\\
&\quad+\big(\|a_1\|_{X_1}+\sqrt t\|b_1\|_{X_2}\big)\big(\|a_2\|_{X_1}+\sqrt t \|b_2\|_{X_2}	\big)\big(\|a_3\|_{X_1}+ \|b_3\|_{X_1}\big)\Big):=R(t)
\end{aligned}
\end{equation*}
which clearly implies 
\begin{equation}\label{tri:est:7}
K(t,T(x_1,x_2,x_3),Y_1,Y_2)\leq R(t).
\end{equation}
Now,  by using  \autoref{K_fun_new} and \autoref{restric} and by taking  the infima over all the sets $\tilde \Omega(x_k)=\{(a_k,b_k)\in \Omega(x_k) \hbox { s.t. } \|a_k\|_{X_1}\leq \|x_k\|_{X_1} \} $  for $k=1,2,3$ in the right-hand side of \eqref{tri:est:7}, we achieve 
\[
\begin{aligned}
K(t,T(x_1,x_2,x_3),Y_1,Y_2)&\leq 3(1+\varepsilon)C_0\left( \|x_1\|_{X_1} K(\sqrt t, x_2) K(\sqrt t, x_3) \right.\\
&\quad+\left.K(\sqrt t, x_1)\left( \|x_2\|_{X_1} K(\sqrt t, x_3)+\|x_3\|_{X_1}K(\sqrt t, x_2)\right)\right).
\end{aligned}
\]
Multiplying by $t^{-(\theta+\gamma)/2}$  the last inequality, taking the $L^r_*(0,\infty)$-norm and using the H\"older inequality with $s,s^\prime$ as conjugate pair, we obtain \eqref{eqn:ts-abstr-interp} by letting $\varepsilon\rightarrow 0$.
\end{proof}

We recall that interpolation theory also provides the following useful characterization of Besov spaces (see for instance  \cite[Theorem 6.2.4]{BL}).

\begin{prop}\label{interp_sob_bes} Let $\Omega\subseteq\R^d$ be a Lipschitz open set. For any $\theta\in(0,1),$ $r,s\in[1,\infty]$ and $\sigma_1\neq \sigma_2\in \Z $,
\begin{equation}
\label{eqn:interp-sob-bes}
\left(W^{\sigma_1,r}(\Omega),W^{\sigma_2,r}(\Omega)\right)_{\theta,s}=B^{(1-\theta) \sigma_1+\theta \sigma_2}_{r,s}(\Omega).
\end{equation}
Moreover, the same holds if we restrict all spaces in \eqref{eqn:interp-sob-bes} to the linear subspace of divergence-free vector fields. 
\end{prop}
Notice that for the sake of simplicity, we did not define, in \autoref{sec:tool}, Besov spaces of order less than or equal to $0$. However, we will apply \autoref{interp_sob_bes} only for the Besov spaces of strictly positive $\theta$. The statement for divergence-free vector fields follows instead from the same proof as \eqref{eqn:interp-sob-bes}, since the construction in the interpolation is based on mollification at a suitable scale, and convolutions preserve the divergence-free structure of the vector fields.

\section{Regularity of Euler equation}\label{sec:euler}
The following result about elliptic equations follows by a direct application of \autoref{bil:thm} and \autoref{thm:tri} of the previous section. The reader can compare the following proposition with \cite[Proposition 3.1]{CD18} obtained for H\"older spaces through estimates on a representation formula for $p$ and $q$.
\begin{prop}\label{prop:bil_tril}
Let $\gamma ,\theta \in (0,1)$ and $r\in (1,\infty)$. Let $u,w,z:\T^3\rightarrow \R^3$ be divergence-free vector fields and let $p,q:\T^3 \rightarrow \R^3$ be the unique $0$-average solutions of 
\begin{align}
-\Delta p&=\diver \diver (u\otimes w),\label{laplace_bil}\\
-\Delta q &= \diver \diver \diver( u \otimes w \otimes z)\label{laplace_tril}.
\end{align}
Then, for any $s \in [1,\infty]$, we have 
\begin{equation}\label{est_p_bil}
\|p\|_{B^{\gamma+\theta}_{r,s}}\leq C \|u\|_{B^{\gamma}_{2r,2s}} \|w\|_{B^{\theta}_{2r,2s}}.
\end{equation}
Furthermore, if $\theta +\gamma> 1$ 
\begin{align}\label{est_q_tril}
\|q\|_{B^{\gamma+\theta-1}_{r,s}}\leq C\bigg( \|u\|_{L^{3r}} \|w\|_{B^{\gamma}_{3r,2s}}\|z\|_{B^{\theta}_{3r,2s}} +\|u\|_{B^{\gamma}_{3r,2s}}\Big( \|w\|_{L^{3r}}\|z\|_{B^{\theta}_{3r,2s}}+\|w\|_{B^{\theta}_{3r,2s}}\|z\|_{L^{3r}} \Big) \bigg).
\end{align}
\end{prop}
\begin{proof} We denote by $W^{1,r}_{\rm div}$ the linear subspace of $W^{1,r}$ made by divergence-free vector fields (and similarly for $B^\theta_{r,s,{\rm div}}$). Let $T(u,w)$ be the operator that for each couple $(u,w)$ associate the unique 0-average solution of \eqref{laplace_bil}. By the Calder\'on-Zygmund theory, we have
\[
\| T(u,w) \|_{L^r} \leq C \|u\|_{L^{2r}} \|w\|_{L^{2r}}.
\]
Moreover since $\diver u =\diver w=0$ the right-hand side of \eqref{laplace_bil} can be rewritten as
\[
\diver \diver (u\otimes w)=\partial^2_{ij}(u^i w^j)=\partial_j( u^i \partial_i w^j)=\partial_j u^i \partial_i w^j,
\]
thus we can use again  Calder\'on-Zygmund  to get
\[
\| T(u,w) \|_{W^{1,r}} \leq C \|u\|_{L^{2r}} \|w\|_{W^{1,2r}}
\]
and 
\[
 \| T(u,w) \|_{W^{2,r}} \leq C \|u\|_{W^{1,2r}} \|w\|_{W^{1,2r}}.
\]
Since, by  \autoref{interp_sob_bes}, we have the embedding $W^{1,r}_{\rm div}\hookrightarrow B^1_{r,\infty,{\rm div}}=(L^r_{\rm div},W^{2,r}_{\rm div})_{\frac{1}{2},\infty}$, we can apply \autoref{bil:thm} with $X_1=L^{2r}_{\rm div}$, $X_2=W^{1,2r}_{\rm div}$, $Y_1=L^r_{\rm div}$, $Y_2=W^{2,r}_{\rm div}$, hence obtaining  \eqref{est_p_bil}. Note that it is important that all the spaces above consist of divergence-free vector fields.

\noindent The proof of \eqref{est_q_tril} follows similarly as a consequence of Calder\'on-Zygmund and  \autoref{thm:tri}, with $X_1=L^{3r}_{\rm div}$, $X_2=W^{1,3r}_{\rm div}$, $Y_1=W^{-1,r}_{\rm div}$ and $Y_2=W^{1,r},_{\rm div}$ once one  notices that the solenoidal nature of $u,w,z$ implies that
\begin{equation*}
\begin{aligned}
\diver \diver \diver (u \otimes w\otimes z)&=\partial^3_{ijk}(u^i w^j z^k)=\partial^2_{ij} (\partial_k u^i w^j z^k)+\partial^2_{ij} (u^i\partial_kw^j z^k)\\
&=\partial_j(\partial_k u^i \partial_i w^j z^k + \partial_k u^i w^j \partial_i z^k) +\partial_i( \partial_ju^i \partial_k w^j z^k + u^i \partial_kw^j \partial_jz^k).\qedhere
\end{aligned}
\end{equation*}
\end{proof}

We consider now a weak  solution $(u,p)$ of the incompressible Euler equations \eqref{E}. Taking the divergence of the first equation in \eqref{E}, using the incompressibility constraint $\diver u=0$, the pressure $p$ solves 
\begin{equation}\label{lapl_p}
-\Delta p=\diver \diver (u\otimes u),
\end{equation}
thus it can be uniquely determined if one imposes that $\int_{\T^3} p(t,x) \, dx=0$,  for any time $t\in (0,T)$. For every $\theta \in (0,1)$ and $r\in (1,\infty)$, a direct application of Calder\'on-Zygmund leads to 
\begin{equation}\label{p_besov}
\|p(t)\|_{B^\theta_{r,\infty}}\leq C \|u(t)\|^2_{B^\theta_{2r,\infty}}.
\end{equation}

\noindent Since our solutions are just weak solutions, we will need to mollify  \eqref{E} in order to justify some computations; moreover, we will use tune the convolution parameter in terms of the time increment $h$ (a similar approach was used for instance in \cite{DKM2007}). By regularizing (in space) the equations \eqref{E}, one gets that the couple $(u_\delta, p_\delta) = (u\ast \varphi_\delta, p\ast \varphi_\delta) $ solves
\begin{equation}\label{E_mol}
\left\{\begin{array}{l}
\partial_t u_\delta + \diver (u_\delta  \otimes u_\delta ) +\nabla p_\delta =\diver R_\delta \\ 
\diver u_\delta  = 0\,
\end{array}\right.,
\end{equation}
where $R_\delta =u_\delta  \otimes u_\delta-(u\otimes u)_\delta$. We can now prove our main theorem.

\begin{proof}[Proof of \autoref{t:main}]
Let $h>0$ be a time increment. When it will help the readability we will also put in the constants $C$ all the norms of $u$ and $p$ which are already known to be finite. We prove the theorem for $s<\infty$, since the case $s=\infty$ is a simple adaptation and it is easier using the identification $B^{\theta}_{\infty,\infty}=C^\theta$.  \\

\noindent {\it Proof of (i)}. Assume that $u\in L^{2s}((0,T);B^\theta_{2r,\infty}(\T^3))$, for some $s\in [1,\infty)$. We split
\begin{equation}\label{split_ut}
\|u(t+h)-u(t)\|_{L^r}\leq \|u(t+h)-u_\delta(t+h)\|_{L^r}+\|u_\delta(t+h)-u_\delta(t)\|_{L^r}+\|u_\delta(t)-u(t)\|_{L^r}.
\end{equation}
Using \eqref{molli:1} we have $\|u_\delta(t)-u(t)\|_{L^r}\leq C\delta^\theta  \|u(t)\|_{B^\theta_{r,\infty}}$ for every $t\in (0,T)$, from which we deduce
\[
\begin{aligned}
\left( \int_0^{T-h} \|u(t+h)-u_\delta(t+h)\|_{L^r}^s \,dt\right)^\frac{1}{s}+\left(  \int_0^{T-h} \|u(t)-u_\delta(t)\|_{L^r}^s \,dt\right)^\frac{1}{s}&\leq C \delta^{\theta} \|u\|_{L^s (B^\theta_{r,\infty})}\\
&\leq  C \delta^{\theta } \|u\|_{L^{2s}(B^\theta_{2r,\infty})}.
\end{aligned}
\]
In the last inequality we used the fact that both the time and spatial domains are bounded. We are left with the second term in the right-hand side of \eqref{split_ut}. Since $u_\delta$ solves \eqref{E_mol}, using also  \eqref{molli:2} and \eqref{p_besov} we get
\begin{equation*}
\begin{split}
\|u_\delta(t+h)-u_\delta(t)\|_{L^r}&\leq \int_t^{t+h}\|\partial_t u_\delta(\tau)\|_{L^r} \,d\tau \leq \int_t^{t+h} \Big( \| \diver (u\otimes u)_\delta(\tau) \|_{L^r} +\| \nabla p_\delta (\tau)\|_{L^r}\Big)\,d\tau
\\& 
\leq C \delta^{\theta-1} \int_t^{t+h} \big(\| u \otimes  u (\tau)\|_{B^{\theta}_{r,\infty}}+  \| p (\tau)\|_{B^{\theta}_{r,\infty}}\big) \, d\tau
\leq C \delta^{\theta-1}\int_t^{t+h} \|u(\tau)\|^2_{B^\theta_{2r,\infty}}\,d\tau.
\end{split}
\end{equation*}
By the  Young's inequality, with respect to the unitary measure $d\tau /h$, we deduce
\[
\|u_\delta(t+h)-u_\delta(t)\|_{L^r}^s\leq C \delta^{(\theta-1)s}h^{s-1}\int_0^T\chi(\tau)_{(t,t+h)} \|u(\tau)\|^{2s}_{B^\theta_{2r,\infty}}\,d\tau,
\]
from which, by integrating in time, we conclude
\begin{equation*}
\begin{split}
\int_0^{T-h} \|u_\delta(t+h)-u_\delta(t)\|^s_{L^r}\,dt&\leq C \delta^{(\theta-1)s} h^{s-1}\int_0^{T-h}\int_0^T \chi(\tau)_{(t,t+h)}\|u(\tau)\|^{2s}_{B^\theta_{2r,\infty}}\,d\tau\,dt
\\&\leq C \delta^{(\theta-1)s} h^{s}\|u\|^{2s}_{L^{2s}(B^\theta_{2r,\infty})},
\end{split}
\end{equation*}
where in the last inequality we also used $\int_0^{T-h} \chi(t)_{(\tau-h,\tau)}\,dt\leq h$. By choosing $\delta =h$, we achieve 
\[
\left( \int_0^{T-h} \|u(t+h)-u(t)\|^s_{L^r}\,dt\right)^{\frac 1s}\leq C h^\theta \left(\|u\|_{L^{2s}(B^\theta_{2r,\infty})}+ \|u\|^2_{L^{2s}(B^\theta_{2r,\infty})}\right),
\]
from which, by taking the supremum all over $h\in (0,T)$, we conclude $u\in B^\theta_{s,\infty}((0,T);L^r(\T^3))$. Since $p$ solves \eqref{lapl_p}, we can use  \eqref{est_p_bil} with $u=w= u(t)$, $\gamma =\theta$, $s=\infty$, getting 
\begin{equation}\label{p_besov2theta}
\|p(t)\|_{B^{2\theta}_{r,\infty}} \leq C \|u(t)\|^2_{B^\theta_{2r,\infty}}.
\end{equation}
Taking the $L^s(0,T)$-norm, we deduce that $p\in L^s((0,T);B^{2\theta}_{r,\infty}(\T^3))$, namely that $(i)$ holds.\\

\noindent {\it Proof of (ii)}. Let $\theta >1/2$ and $\beta \in [0,2\theta-1)$. Note that 
\[
-\Delta (p(t+h)-p(t)) =\diver \diver \Big( \big(u(t+h)-u(t)\big) \otimes u(t+h) + u(t)\otimes \big( u(t+h)-u(t)\big)\Big).
\]
Thus, by using \eqref{est_p_bil} with $\gamma=1-\theta+\beta$, $s=\infty$, we get 
\begin{equation}\label{est_p_1}
\|p(t+h)-p(t)\|_{B^{1+\beta}_{r,\infty}}\leq C \|u(t+h)-u(t)\|_{B^{1-\theta+\beta}_{2r,\infty}} \left( \|u(t+h)\|_{B^\theta_{2r,\infty}}+\|u(t)\|_{B^\theta_{2r,\infty}}\right),
\end{equation}
and taking the $L^s(0,T-h)$-norm in time, by also using the  H\"older inequality, we achieve
\begin{equation}\label{est_p_2}
\left( \int_0^{T-h}\|p(t+h)-p(t)\|^q_{B^{1+\beta}_{r,\infty}}\,dt \right)^{\frac{1}{s}}\leq C \left( \int_0^{T-h} \|u(t+h)-u(t)\|^{\frac{3s}{2}}_{B^{1-\theta+\beta}_{2r,\infty}}\,dt\right)^{\frac{2}{3s}} \|u\|_{L^{3s}(B^\theta_{2r,\infty})}.
\end{equation}
By the interpolation inequality \eqref{interp_besov}, the H\"older inequality, and since $u\in B^\theta_{\frac{3s}{2},\infty}((0,T);L^{2r}(\T^3))$ by $(i)$, we can estimate
\begin{align*}
\int_0^{T-h} &\|u(t+h)-u(t)\|^\frac{3s}{2}_{B^{1-\theta+\beta}_{2r,\infty}}\,dt  \leq  \int_0^{T-h} \|u(t+h)-u(t)\|^{\frac{3s}{2}\frac{2\theta-1-\beta}{\theta }} _{L^{2r}}\|u(t+h)-u(t)\|^{\frac{3s}{2}\frac{1-\theta+\beta }{\theta}} _{B^{\theta}_{2r,\infty}} \,dt \\
&\leq \left(  \int_0^{T-h} \|u(t+h)-u(t)\|^\frac{3s}{2} _{L^{2r}}\,dt\right)^{\frac{2\theta-1-\beta}{\theta}}\left(  \int_0^{T-h} \|u(t+h)-u(t)\|^\frac{3s}{2} _{B^\theta_{2r,\infty}}\,dt\right)^{\frac{1-\theta+\beta}{\theta}}\\
&\leq C h^{\frac{3s}{2}(2\theta-1-\beta)} \|u\|^{\frac{3s}{2}\frac{2\theta-1-\beta}{\theta}}_{B^\theta_{\frac{3s}{2},\infty}(L^{2r})}\|u\|^{\frac{3s}{2}\frac{1-\theta+\beta}{\theta}}_{L^{\frac{3s}{2}}(B^\theta_{2r,\infty})}\leq C h^{\frac{3s}{2}(2\theta-1-\beta)}  .
\end{align*}
By plugging this last estimate in \eqref{est_p_2}, we conclude that  $p\in B^{2\theta-1-\beta}_{s,\infty}((0,T);B^{1+\beta}_{r,\infty}(\T^3))$, since we get
\[
\left( \int_0^{T-h}\|p(t+h)-p(t)\|^s_{B^{1+\beta}_{r,\infty}}\,dt \right)^{\frac{1}{s}}\leq C h^{2\theta-1-\beta}.
\]

\noindent {\it Proof of (iii)}. In order to prove the Besov regularity in time of the pressure, we split 
\begin{equation}\label{split_p}
\|p(t+h)-p(t)\|_{L^r}\leq \|p(t+h)-p_\delta(t+h)\|_{L^r}+\|p_\delta(t+h)-p_\delta(t)\|_{L^r}+\|p_\delta(t)-p(t)\|_{L^r}.
\end{equation}
Using \eqref{molli:1} and \eqref{p_besov2theta}, we have, for every $t\in (0,T),$
\[
\|p_\delta(t)-p(t)\|_{L^r}\leq C\delta^{2\theta}\|p(t)\|_{B^{2\theta}_{r,\infty}}\leq C\delta^{2\theta}\|u(t)\|^2_{B^\theta_{2r,\infty}}\leq C\delta^{2\theta}\|u(t)\|^2_{B^\theta_{3r,\infty}},
\]
 from which we deduce
\[
\int_0^{T-h} \|p(t+h)-p_\delta(t+h)\|_{L^r}^s \,dt+ \int_0^{T-h} \|p(t)-p_\delta(t)\|_{L^r}^s \,dt \leq  C \delta^{2\theta s} \|u\|^{2s}_{L^{3s}(B^\theta_{3r,\infty})}.
\]
It remains to prove the estimate for the middle term $\|p_\delta(t+h)-p_\delta(t)\|_{L^r}$ in the right-hand side of \eqref{split_p}.  Notice that  $p_\delta(t+h)-p_\delta(t)$ solves 
\begin{equation*}
\begin{aligned}
-\Delta ( &p_\delta(t+h) -p_\delta(t) ) = \diver \diver\Big( R_\delta(t) - R_\delta(t+h) +  u_\delta(t+h) \otimes u_\delta(t+h) - u_\delta(t) \otimes u_\delta(t) \Big)
\\&= \diver \diver \Big( R_\delta(t) - R_\delta(t+h) + \int_{t}^{t+h} \Big( \frac{d}{d\tau}u_\delta(\tau,x) \otimes u_\delta(\tau,x)+ u_\delta(\tau,x) \otimes  \frac{d}{d\tau} u_\delta(\tau,x) \Big) \, d\tau \Big)
\\&= \diver \diver\Big( R_\delta(t) - R_\delta(t+h) +
 \int_t^{t+h} \Big( (\diver(u_\delta \otimes u_\delta)- \nabla p_\delta- \diver R_\delta) \otimes u_\delta
 \\&\quad+ u_\delta \otimes  (\diver(u_\delta \otimes u_\delta)- \nabla p_\delta- \diver R_\delta)  \Big) \, d\tau\Big)\,.
\end{aligned}
\end{equation*}
Thus $p_\delta(t+h)-p_\delta(t)=q^1+q^2+q^3,$ where $q^1,q^2,q^3$ are the unique $0$-average solutions to 
\begin{equation*}
\begin{aligned}
-\Delta q^1&=\diver \diver( R_\delta(t,x) - R_\delta(t+h,x)),\\
\Delta q^2&= 2 \int_t^{t+h} \diver\diver( (\diver R_\delta+ \nabla p_\delta) \otimes u_\delta) \, d\tau,\\
-\Delta q^3 &= \int_t^{t+h} \diver \diver \diver(u_\delta \otimes u_\delta \otimes u_\delta ) \, d\tau.
\end{aligned}
\end{equation*}
By Calder\'on-Zygmund, \eqref{molli:3} and \eqref{molli:2} we have that 
\begin{equation*}
\|q^1(t)\|_{L^r}\leq C\big(\|R_\delta(t+h)\|_{L^r}+\|R_\delta(t)\|_{L^r}\big)\leq C\delta^{2\theta}\big( \|u(t+h)\|^2_{B^\theta_{3r,\infty}}+ \|u(t)\|^2_{B^\theta_{3r,\infty}}\big),
\end{equation*}
and
\[
\|q^2(t)\|_{L^r}\leq C \int_t^{t+h} \Big( \| \diver R_\delta(\tau)\|_{L^\frac{3r}{2}}+\|\nabla p_\delta (\tau)\|_{L^\frac{3r}{2}}\Big) \|u_\delta(\tau)\|_{L^{3r}}\,d\tau \leq C \delta^{2\theta-1}\int_t^{t+h}\|u(\tau)\|^3_{B^\theta_{3r,\infty}}\,d\tau.
\]
Hence, by taking the $L^s(0,T-h)$-norm, we deduce
\begin{equation}\label{est:q1}
\|q^1\|_{L^s(L^r)}\leq C \delta^{2\theta} \|u\|^2_{L^{3s}(B^\theta_{3r,\infty})}.
\end{equation}
and by the Young inequality we have
\begin{align}\label{est:q2}
\int_0^{T-h}\|q^2(t)\|_{L^r}^s\,dt &\leq C\delta^{(2\theta-1)s} h^{s-1}\int_0^{T-h}\left( \int_0^T \chi_{(t,t+h)}(\tau)\|u(\tau)\|^{3s}_{B^\theta_{3r,\infty}}\,d\tau\right)\,dt\nonumber \\
&\leq C\delta^{(2\theta-1)s} h^s\|u\|^{3s}_{L^{3s}(B^\theta_{3r,\infty})}.
\end{align}
For $q^3$ we can use, for any $\varepsilon>0$, \eqref{est_q_tril} with $\theta=\gamma=(1+\varepsilon)/2$, $s=\infty$, $u=w=z=u_\delta(t)$, 
getting
\begin{equation}\label{est:q^3_1}
\|q^3(t)\|_{L^r}\leq \|q^3(t)\|_{B^\varepsilon_{r,\infty}}\leq C \int_t^{t+h} \|u_\delta(\tau)\|_{L^{3r}}\|u_\delta(\tau)\|^2_{B^{\frac{1+\varepsilon}{2}}_{3r,\infty}}\,d\tau .
\end{equation}
By \eqref{interp_besov_1} and the estimate \eqref{molli:2}, we have
\[
\|u_\delta (t)\|_{B^{\frac{1+\varepsilon}{2}}_{3r,\infty}}\leq \|u_\delta(t)\|_{B^\theta_{3r,\infty}}^\frac{1-\varepsilon}{2(1-\theta)} \|u_\delta(t)\|_{W^{1,3r}}^\frac{1+\varepsilon-2\theta}{2(1-\theta)}\leq C \delta^{\theta-\frac{1+\varepsilon}{2}} \|u(t)\|_{B^\theta_{3r,\infty}}.
\]
Plugging this last estimate in \eqref{est:q^3_1}, we achieve
\[
\|q^3(t)\|_{L^r}\leq C \delta^{2\theta-1-\varepsilon}\int_t^{t+h} \|u(\tau)\|^3_{B^\theta_{3r,\infty}}\,d\tau,
\]
from which we deduce
\begin{equation}\label{est:q3}
\|q^3\|_{L^s(L^r)}\leq C \delta^{2\theta-1-\varepsilon}h \|u\|^3_{L^{3s}(B^\theta_{3r,\infty})}.
\end{equation}
Choosing $\delta=h$, from \eqref{est:q1}, \eqref{est:q2} and \eqref{est:q3}, we conclude
\[
\left( \int_0^{T-h}\|p_\delta(t+h)-p_\delta(t)\|_{L^r}^s\,dt \right)^{\frac{1}{s}}\leq C h^{2\theta-\varepsilon}\left( \|u\|^2_{L^{3s}(B^\theta_{3r,\infty})}+\|u\|^3_{L^{3s}(B^\theta_{3r,\infty})} \right),
\]
which implies that $p\in B^{2\theta-\varepsilon}_{s,\infty}((0,T);L^r(\T^3))$. If now $\theta>1/2$, we have to  prove that  $ p\in W^{1,s}((0,T);B^{2\theta-1}_{r,\infty}(\T^3))$. It is enough to show that $\partial_t p \in L^s((0,T);B^{2\theta-1}_{r,\infty}(\T^3))$. Thus we can write, by using \eqref{dert_p},  $\partial_t p=q^1+q^2$ where $q^1,q^2$ are the unique $0$-average solutions of 
\begin{equation*}
\begin{aligned}
-\Delta q^1&=\diver\diver\diver (u\otimes u \otimes u),\\
\Delta q^2& =2\diver\diver(\nabla p\otimes u).
\end{aligned}
\end{equation*}
Since, by \eqref{p_besov2theta},
\[
\|\nabla p(t)\|_{B^{2\theta-1}_{\frac{3r}{2},\infty}}\leq C \|u(t)\|^2_{B^\theta_{3r,\infty}},
\]
by Calder\'on-Zygmund 
 we get
\[
\|q^2(t)\|_{B^{2\theta-1}_{r,\infty}}\leq C \|(\nabla p \otimes u)(t) \|_{B^{2\theta-1}_{r,\infty}}\leq C \|\nabla p(t)\|_{B^{2\theta-1}_{\frac{3r}{2},\infty}} \|u(t)\|_{B^\theta_{3r,\infty}}\leq C \|u(t)\|^3_{B^\theta_{3r,\infty}}.
\]
Moreover, by \eqref{est_q_tril} with $\gamma=\theta,$ $s=\infty$ and $u=w=z=u(t),$
\[
\|q^1(t)\|_{B^{2\theta-1}_{r,\infty}}\leq C\|u(t)\|^3_{B^\theta_{3r,\infty}}.
\]
Hence, by taking the $L^s(0,T)$-norm we obtain
\[
\|\partial_t p\|_{L^s(B^{2\theta-1}_{r,\infty})}\leq\|q^1\|_{L^s(B^{2\theta-1}_{r,\infty})}+\|q^2\|_{L^s(B^{2\theta-1}_{r,\infty})} \leq C\|u\|_{L^{3s}(B^{\theta}_{3r,\infty})}^3,
\]  
which concludes the proof of $(iii)$.\\

\noindent {\it Proof of (iv)}. By  \autoref{lemma_dtp} we have that $\partial_t p$ solves \eqref{dert_p}. Therefore $\partial_t p(t+h)-\partial_t p(t)=q^1+q^2$ where 
\begin{equation*}
\begin{aligned}
\Delta q^1 &= \diver \diver \diver (u(t+h) \otimes u(t+h)\otimes u(t+h)-u(t) \otimes u(t)\otimes u(t))
\\&= \diver \diver \diver \big(( u(t+h)-u(t)) \otimes u(t+h)\otimes u(t+h)+ u(t) \otimes( u(t+h)-u(t))\otimes u(t+h)\\
&\quad \,+ u(t) \otimes u(t)\otimes (u(t+h)- u(t))\big),
\end{aligned}
\end{equation*}
\begin{equation*}
\Delta q^2= 2\diver \diver \big(\nabla p(t+h) \otimes u(t+h)-\nabla p(t) \otimes u(t) \big).
\end{equation*}
To estimate $q^1,$ for any small $\varepsilon>0,$ we apply \eqref{est_q_tril} with $\gamma=1-\theta+\varepsilon$ and $s=\infty,$   in such a way that the factor $u(t+h)-u(t)$ gets  only the $B^{1-\theta+\varepsilon}_{3r,\infty}$-norm and not the $B^{\theta}_{3r,\infty}$-norm. Thus we get 
\begin{equation*}
\begin{aligned}
\|q^1(t)\|_{L^r}&\leq \|q^1(t)\|_{B^{\varepsilon}_{r,\infty}}\leq C\|u(t+h)-u(t)\|_{B^{1-\theta+\varepsilon}_{3r,\infty}}\big(\|u(t+h)\|^2_{B^{\theta}_{3r,\infty}}+\|u(t)\|^2_{B^{\theta}_{3r,\infty}}\big).
\end{aligned}
\end{equation*}
Integrating in time on $(0,T-h)$ yields to 
\begin{equation*}
\begin{aligned}
\int_0^{T-h}\|q^1(t)\|^{s}_{L^r}\,dt&\leq C\int_0^{T-h}\|u(t+h)-u(t)\|^{s}_{B^{1-\theta+\varepsilon}_{3r,\infty}}\big(\|u(t+h)\|^{2s}_{B^{\theta}_{3r,\infty}}+\|u(t)\|^{2s}_{B^{\theta}_{3r,\infty}}\big)\,dt
\end{aligned}
\end{equation*}
and by the Cauchy-Schwarz inequality we get 
\begin{equation*}
\begin{aligned}
\int_0^{T-h}\|q^1(t)\|^s_{L^r}\,dt&\leq C\left(\int_0^{T-h}\|u(t+h)-u(t)\|^{2s}_{B^{1-\theta+\varepsilon}_{3r,\infty}}\,dt\right)^{\frac12}\|u\|_{L^{4s}(B^\theta_{3r,\infty})}^{2s}.
\end{aligned}
\end{equation*}
Now, by \eqref{interp_besov} together with the H\"older inequality in time, we have
\begin{align*}
\int_0^{T-h} \|u(t+h)-&u(t)\|^{2s} _{B^{1-\theta+\varepsilon}_{3r,\infty}}\,dt  \leq  \int_0^{T-h} \|u(t+h)-u(t)\|^{2s\frac{2\theta-1-\varepsilon}{\theta}} _{L^{3r}}\|u(t+h)-u(t)\|^{2s\frac{1-\theta+\varepsilon }{\theta}} _{B^{\theta}_{3r,\infty}} \,dt \\
&\leq \left(  \int_0^{T-h} \|u(t+h)-u(t)\|^{2s} _{L^{3r}}\,dt\right)^{\frac{2\theta-1-\varepsilon}{\theta}}\left(  \int_0^{T-h} \|u(t+h)-u(t)\|^{2s} _{B^\theta_{3r,\infty}}\,dt\right)^{\frac{1-\theta+\varepsilon}{\theta}}\\
&\leq C h^{2s(2\theta-1-\varepsilon)}\|u\|_{B^{\theta}_{2s,\infty}(L^{3r})}^{2s\frac{2\theta-1-\varepsilon}{\theta}} \|u\|^{2s{\frac{1-\theta+\varepsilon}{\theta}}}_{L^{2s}(B^\theta_{3r,\infty})}\leq C h^{2s(2\theta-1-\varepsilon)},
\end{align*}
where in the last inequality we used $u\in B^{\theta}_{3s,\infty}((0,T);L^{3r}(\T^3))\hookrightarrow  B^{\theta}_{2s,\infty}((0,T);L^{3r}(\T^3))$, that comes from $(i)$. Thus we conclude with 
\begin{equation}\label{est:q1:bis}
\int_0^{T-h}\|q^1(t)\|^{s}_{L^r}\,dt\leq C h^{s(2\theta-1-\varepsilon)}.
\end{equation}
Similarly, we obtain
\begin{equation}\label{est:q2:bis}
\begin{aligned}
\int_0^{T-h}\|q^2(t)\|_{L^r}^s\,dt & \leq C\int_0^{T-h}\|(\nabla p\otimes u)(t+h)-\nabla p\otimes u)(t)\|_{L^r}^s\,dt\leq Ch^{s(2\theta-1-\varepsilon)}\|\nabla p \otimes u\|_{B^{2\theta-1-\varepsilon}_{s,\infty}(L^r)}^{s}\\
&\leq Ch^{s(2\theta-1-\varepsilon)} \left(\|\nabla p\|_{B^{2\theta-1-\varepsilon}_{2s,\infty}(L^{2r})}\|u\|_{B^{2\theta-1-\varepsilon}_{2s,\infty}(L^{2r})} \right)^s \\
&\leq Ch^{s(2\theta-1-\varepsilon)}\left(\|\nabla p\|_{B^{2\theta-1-\varepsilon}_{2s,\infty}(L^{2r})}\|u\|_{B^{\theta}_{2s,\infty}(L^{2r})} \right)^s\leq C h^{s(2\theta-1-\varepsilon)},
\end{aligned}
\end{equation}
where we used that $u\in B^{\theta}_{2s,\infty}((0,T);L^{2r}(\T^3))$ by $(i)$, and $\nabla p \in  B^{2\theta-1-\varepsilon}_{2s,\infty}((0,T);L^{2r}(\T^3)) $ by $(ii)$.
Summing up  \eqref{est:q1:bis} and \eqref{est:q2:bis} we obtain $\partial_t p \in B^{2\theta-1-\varepsilon}_{s,\infty}((0,T);L^r(\T^3)),$ as desired.
\end{proof}

\begin{lemma}\label{lemma_dtp}
Let $u\in L^{3s}((0,T);B^\theta_{3r,\infty}(\T^3))$ for some $r,s \in [1,\infty]$ and $\theta\in (1/2,1)$. Then $\partial_t p$ solves 
\begin{equation}\label{dert_p}
-\Delta \partial_t p= \diver \diver \diver (u\otimes u \otimes u) + 2\diver \diver (\nabla p \otimes u),
\end{equation}
in the distributional sense.
\end{lemma}
\begin{proof}
For every $\delta >0$, we denote by $p^\delta$ the unique $0$-average  solution of
\[
-\Delta p^\delta= \diver \diver (u_\delta \otimes u_\delta).
\]
Note that by Calder\'on-Zygmund,  $p^\delta \rightarrow p$ in $L^{\frac{3s}{2}}((0,T);L^{\frac{3r}{2}}(\T^3))$ as $\delta  \to 0$. Thus $\partial_t p^\delta \to \partial_t p$ in distribution.
Since $\partial_t u_\delta\in L^{\frac{3s}{2}}((0,T);C^\infty(\T^3))$ from \eqref{E_mol}, we can compute 
\begin{align*}
\partial_t \diver 	\diver (u_\delta \otimes u_\delta)&= 2 \diver \diver (\partial_t u_\delta\otimes u_\delta)= \diver \diver \diver ( u_\delta\otimes u_\delta\otimes u_\delta)\\
&-2 \diver \diver (\nabla p_\delta \otimes u_\delta) + 2 \diver 	\diver (\diver R_\delta \otimes u_\delta).
\end{align*}
Obviously $u_\delta \to u$ in $L^{3s}((0,T);L^{3r}(\T^3))$. By \eqref{molli:3}, since $\theta>1/2$ we have that $\diver R_\delta \to 0$ in $L^{\frac{3s}{2}}((0,T);L^{\frac{3r}{2}}(\T^3) ) $. Moreover by $(i)$ in \autoref{t:main} we also have 
$\nabla p_\delta \to \nabla p$ in $L^{\frac{3s}{2}}((0,T);L^{\frac{3r}{2}} (\T^3)) $. Thus we conclude that in the distributional sense
\[
\partial_t \diver 	\diver (u_\delta \otimes u_\delta)\to \diver \diver \diver ( u \otimes u \otimes u)-2 \diver \diver (\nabla p \otimes u).\qedhere
\]
\end{proof}

\begin{remark}
In the above proof, one can make explicit quantitative estimates on the quantities which appear in the statement of  \autoref{t:main}. For instance, as regards $(i)$ we have 
\begin{align*}
\|u\|_{B^\theta_{s,\infty}(L^r)}& \leq C\left( \|u\|_{L^{s}(B^\theta_{r,\infty})} +
\|u\|^2_{L^{2s}(B^\theta_{2r,\infty})}\right),\\
\|p\|_{L^{s}(B^{2\theta}_{r,\infty})} &\leq C\|u\|^2_{L^{2s}(B^\theta_{2r,\infty})} 
\end{align*}
for a constant $C>0$ depending only on $r, s, \theta$.
\end{remark}

\begin{remark}[The case $r=1$] When $r=1$, the statements $(i)$ and $(ii)$ of 
\autoref{t:main} on the pressure may not be true in general. On the positive side, if $u \in L^{3s}((0,T);W^{1,1}(\T^3))$, the compensated compactness methods \cite{CLMS} give that the pressure belongs to  $ L^{\frac{3s}{2}}((0,T);W^{2,1}(\T^3))$ (namely, the result with $r=1$ and $\theta=1$ would hold). On the other side, however, if $r=1$ and $\theta=0$, the lack of the Calder\'on-Zygmund theory gives us that a solution $p$ to \eqref{eqn:p} is in general not more than in the weak-$L^1(\T^3)$ space. Trying to repeat the proof of the abstract interpolation result of \autoref{bil:thm}, as we did in \autoref{prop:bil_tril} for $r=1$, this constitutes a problem because we would need to apply the interpolation result with  $Y_1=L^1_{\rm weak, \, div}$, $Y_2=W^{2,r}_{\rm div}$. 
 Hence, \autoref{bil:thm} would only give us that $p(t) \in (L^{1}_{\rm weak}(\T^3) , W^{2,1}(\T^3))_{\theta, r}$ and it is unclear if such space would coincide with a suitable Besov-type space.
\end{remark}

\begin{proof}[Proof of \autoref{t:main-easy}] The proof is just a consequence of $(i)$, $(ii)$ and $(iv)$ of  \autoref{t:main} together with the embeddings
$W^{\theta,r}\hookrightarrow B^\theta_{r,\infty}\hookrightarrow W^{\gamma,r},$
that hold true for any $r\in [1,\infty]$ and $\theta, \gamma \in (0,1)$ with $\theta>\gamma$.
\end{proof}

\textbf{ Acknowledgements}. The first two authors have been supported by
the SNSF Grant 182565 "Regularity issues for the Navier-Stokes equations and for other variational problems".


\end{document}